\documentclass[11pt]{article}

\usepackage{amsfonts,amsmath,amssymb,cite}
\usepackage{amsmath,amsthm}
\usepackage{enumerate}
\usepackage{latexsym}
\usepackage[T1]{fontenc}
\usepackage{color}
\usepackage[dvipsnames]{xcolor}
\usepackage{tikz}
\usetikzlibrary{decorations.pathmorphing}
\usetikzlibrary{shapes.geometric}
\usetikzlibrary{svg.path}
\usepackage{etoolbox}
\usepackage{mathtools}

\usepackage{graphicx}
\usepackage{lineno}
\usepackage{tkz-graph}
\usepackage{tkz-berge}
\usepackage{hyperref}
\usepackage{color}

\hypersetup{colorlinks=true}

\hypersetup{colorlinks=true, linkcolor=blue, citecolor=blue,urlcolor=blue}

\newtheorem{thm}{Theorem}

\newtheorem{lem}[thm]{Lemma}

\newtheorem{cor}[thm]{Corollary}
\newtheorem{prop}[thm]{Proposition}

\newtheorem{prob}{Problem}

\theoremstyle{definition}

\theoremstyle{remark}
\newtheorem{remark}[thm]{Remark}

\newcommand{\ml}{\mu^{-}}
\newcommand{\tml}{\mu^{-}_t}
\DeclareMathOperator{\gp}{gp}

\newcommand{\cp}{\,\square\,}

\tikzstyle{vertex}=[circle, draw, inner sep=0pt, minimum size=6pt]

\begin{document}

\title{Lower (total) mutual-visibility number in graphs}
\author{
Bo\v{s}tjan Bre\v{s}ar$^{a,b}$
\thanks{Email: \texttt{bostjan.bresar@um.si}}
\and
Ismael G. Yero$^{c} $\thanks{Email: \texttt{ismael.gonzalez@uca.es}}
}

\maketitle

\begin{center}
$^a$ Faculty of Natural Sciences and Mathematics, University of Maribor, Slovenia\\
$^b$ Institute of Mathematics, Physics and Mechanics, Ljubljana, Slovenia \\
$^c$ Departamento de Matem\'aticas, Universidad de C\'adiz, Algeciras Campus, Spain\\
\end{center}

\begin{abstract}
Given a graph $G$, a set $X$ of vertices in $G$ satisfying that between every two vertices in $X$ (respectively, in $G$) there is a shortest path whose internal vertices are not in $X$ is a mutual-visibility (respectively, total mutual-visibility) set in $G$.  The cardinality of a largest (total) mutual-visibility set in $G$ is known under the name (total) mutual-visibility number, and has been studied in several recent works.

In this paper, we propose two lower variants of these concepts, defined as the smallest possible cardinality among all maximal (total) mutual-visibility sets in $G$, and denote them by $\mu^{-}(G)$ and $\mu_t^{-}(G)$, respectively. While the total mutual-visibility number is never larger than the mutual-visibility number in a graph $G$, we prove that both differences $\mu^{-}(G)-\mu_t^{-}(G)$ and $\mu_t^{-}(G)-\mu^{-}(G)$ can be arbitrarily large. We characterize graphs $G$ with some small values of $\mu^{-}(G)$ and $\mu_t^{-}(G)$, and prove a useful tool called the Neighborhood Lemma, which enables us to find upper bounds on the lower mutual-visibility number in several classes of graphs. We compare the lower mutual-visibility number with the lower general position number, and find a close relationship with the Bollob\'{a}s-Wessel theorem when this number is considered in Cartesian products of complete graphs. Finally, we also prove the NP-completeness of the decision problem related to $\mu_t^{-}(G)$.
\end{abstract}



\section{Introduction} \label{sec:intro}

While studying graph invariants, one aspires to find an extremal (minimum or maximum cardinality) set, which satisfies the defining properties of a given invariant. In most cases, the invariants are hard to be determined in general, while maximal or minimal sets with respect to set inclusion are relatively easy to obtain using a greedy approach. In this way, upper or lower versions of the studied invariants naturally appear, representing the worst outcome of a greedy procedure that satisfies the conditions imposed in the definition of the studied invariant.

For instance, consider a maximum independent set of a graph $G$, which corresponds to the graph invariant denoted by $\alpha(G)$. We could (somewhat naively) attempt to obtain such a set by adding vertices that are pairwise non-adjacent one at a time as long as this is possible. This process eventually produces a maximal independent set, and the cardinality of a smallest possible such set is denoted by $i(G)$ (this invariant is usually called the independent domination number, but sometimes it is also referred to as the lower independence number of $G$). There is some hope that $i(G)$ is a good approximation for $\alpha(G)$, which is precisely the motivation behind introducing the well-covered graphs~\cite{Pl}, which are the graphs $G$ with  $i(G)=\alpha(G)$. Another example comes from graph domination, where the domination number $\gamma(G)$ is the smallest cardinality among all dominating sets in a graph $G$, while the upper domination number $\Gamma(G)$ is the largest cardinality among all minimal dominating sets in $G$ (see~\cite{HaHeHe} for a recent monograph on graph domination).

The concept of mutual-visibility was recently introduced by Di Stefano in~\cite{DiS} with a primary motivation to achieve confidential communication between mobile entities in a network; see~\cite{DiS} for further references containing several other earlier studies. The paper was followed by more studies that all arose within the last year~\cite{CDK, CDDH,CDKY, TK}. Given a graph $G$ and $X\subset V(G)$, two vertices $a,b\in V(G)$ are $X$-{\em visible} if there exists a shortest $a,b$-path $P$ in $G$ such that $V(P)\cap X\subseteq\{a,b\}$. If each pair of vertices in $X$ are $X$-visible, then $X$ is a {\em mutual-visibility} set of $G$. The cardinality of a largest mutual-visibility set in $G$ is the {\em mutual-visibility number}, $\mu(G)$, of $G$. While studying mutual-visibility in strong products of graphs, the authors of~\cite{CDKY} encountered the following useful and natural variation. A set $X$ of vertices in $G$ is a {\em total mutual-visibility set} if every two vertices in $G$ (not only in $X$!) are $X$-visible. The cardinality of a largest total mutual-visibility set in $G$ is the {\em total mutual-visibility number}, $\mu_t(G)$, of $G$.

We initiate here the study of two variations of the two concepts above, which are in line with the initial discussion on upper/lower versions of graph invariants. In particular, since the decision problems of determining $\mu(G)$ and $\mu_t(G)$ are NP-complete~\cite{CDDH}, it is interesting to obtain lower bounds and potential approximations of the studied invariants. A set $X\subset V(G)$ is a {\em maximal (total) mutual-visibility set} in $G$ if $X$ is a (total) mutual-visibility set in $G$ and every set $Y$, with $X\subsetneq Y$, is not a (total) mutual-visibility set in $G$. The cardinality of a smallest maximal (total) mutual-visibility set is the {\em lower (total) mutual-visibility number} of $G$, denoted by $\ml(G)$, respectively, $\tml(G)$.

\subsection{Preliminaries and notation}

In this paper, we only consider simple and undirected graphs. In addition, since  vertices from different connected components cannot belong to the same mutual-visibility set, we will mainly restrict our attention to connected graphs.
Let $G=(V(G),E(G))$ be a connected graph.  A vertex $v$ of $G$ is a {\em cut-vertex} if $G-v$ (the graph obtained from $G$ by removing the vertex $v$ and the edges incident with it) is disconnected, and an edge $e\in E(G)$ is a {\em cut-edge} in $G$ if $G-e$ (the graph obtained from $G$ by removing the edge $e$) is disconnected. It is well known and easy to see that an edge $e$ is a cut-edge in $G$ if and only if $e$ does not lie in a cycle. The {\em neighborhood}, $N_G(v)$, of a vertex $v$ in $G$ is the set of vertices that are adjacent to $v$, and the {\em closed neighborhood} of $v$ is defined as $N_G[v]=N_G(v)\cup\{v\}$. The {\em degree}, $\deg_G(v)$, of $v$ in $G$ is $|N_G(v)|$, while $\Delta(G)$ and $\delta(G)$ denote the maximum, resp. minimum, degree of vertices in $G$. Let $n(G)=|V(G)|$.
By a {\em clique} in $G$ we mean a maximal complete subgraph in $G$; that is, a complete subgraph, which is not properly included in another complete subgraph of $G$. The cardinality of a largest clique in $G$ is denoted by $\omega(G)$.  Additionally, the complete graph, $K_n$, may also be called a clique.
A vertex whose neighborhood induces a complete graph is {\em simplicial}. A graph $G$ is {\em chordal} if any induced cycle in $G$ is a triangle. It is well known that every chordal graph contains a simplicial vertex. A {\em cograph} is a graph that does not contain a path $P_4$ as an induced subgraph. Cographs have been characterized by a procedure that starts with a single vertex and uses operations of complementation and disjoint union; see~\cite{bls-book}. It is well known that every non-trivial cograph contains two vertices that have the same neighborhoods. 

 For graph-theoretic notions not defined in the paper, the reader is referred to the book~\cite{we}.

The {\em distance}, $d_G(u,v)$ between two vertices $u$ and $v$ in a graph $G$ is the length of a shortest $u,v$-path (that is, the number of edges on such a path).
The {\em interval}, $I_G(u,v)$, between $u$ and $v$ in $G$ is the set of all vertices of $G$ that lie on a shortest path in $G$, that is, $I_G(u,v)=\{w\in V(G):\, d_G(u,v)=d_G(u,w)+d_G(w,v)\}$. A subgraph $H$ in a graph $G$ is {\em convex} if every shortest path between vertices of $V(H)$ lies in $H$. In other words, $H$ is convex in $G$ if $I_G(u,v)\subset V(H)$ for any $u,v\in V(H)$. If $H$ is a subgraph of $G$ such that  $d_H(u,v)=d_G(u,v)$ holds for any $u,v\in V(H)$, then $H$ is an {\em isometric subgraph} of $G$.

Given two graphs $G$ and $H$, the {\em Cartesian product} $G\cp H$ of $G$ and $H$ is the graph with $V(G\cp H)=V(G)\times V(H)$ and $(g,h)(g',h')\in E(G\cp H)$ whenever ($g=g'$ and $hh'\in E(H)$) or ($gg'\in E(G)$ and $h=h'$). Given the vertices $g\in V(G)$ and $h\in V(H)$, the set $\{(g,y):\, y\in V(H)\}$ is an {\em $H$-fiber}, and the set $\{(x,h):\, x\in V(G)\}$ is a {\em $G$-fiber} of the Cartesian product $G\cp H$. Clearly, $G\cp H$ has $|V(G)|$ $H$-fibers each of which is isomorphic to $H$. The Cartesian product is associative and commutative, so the Cartesian product $G_1\cp \cdots\cp G_k$ of $k$ graphs $G_1,\ldots,G_k$ is well defined. In particular, the {\em $k$-cube}, $Q_k$, or the {\em hypercube of dimension $k$}, is the Cartesian product of $k$ copies of the graph $K_2$.

\subsection{Goal and organization of the paper}

In this paper, we initiate the study of lower variants of mutual-visibility and total mutual-visibility. Note that the corresponding graph invariants ($\ml(G)$ and $\tml(G)$) give natural lower bounds on the (total) mutual-visibility number in graphs. In addition, they represent the worst outcome of the procedure in which we construct a (total) mutual-visibility set by using a greedy approach adding vertices to the set $S$ until $S$ is a maximal (total) mutual-visibility set. Since the latter procedure can be done in polynomial time, there is an additional reason for studying these invariants.

In Section~\ref{sec:related}, we give some additional arguments for studying the lower mutual-visibility number.
In Section~\ref{ss:neighborhood}, we prove a useful result, called the Neighborhood Lemma, and present some applications yielding upper bounds on the lower mutual-visibility number in several graph classes. Since the introduction of mutual-visibility was inspired by the general position problem in graphs, it is interesting to mention that a lower version of the general position number was recently introduced, and we compare it with the lower mutual-visibility number in Section~\ref{ss_generalposition}. Then, in Section~\ref{ss:bolobas}, we notice an interesting relationship of the lower mutual-visibility number with an old problem of Erd\H{o}s, Hajnal and Moon \cite{EHM}, which was independently solved by Wessel and Bollob\'{a}s. The solution enables us to determine the lower mutual-visibility number of the Cartesian product of two complete graphs. In Section~\ref{sec:complexity}, we establish the NP-completeness of the decision problem regarding the lower total mutual-visibility number (unfortunately, we were not able to determine the same for the lower mutual-visibility number). In the subsequent section, we present general upper and lower bounds on $\ml(G)$ and $\tml(G)$, where $G$ is a connected graph. In particular, we characterize the graphs $G$ with $\ml(G)=2$ as the graphs having a cut-edge. In Section~\ref{sec:two-parameters}, we then show that the differences $\ml(G)-\tml(G)$ and $\tml(G)-\ml(G)$ can be arbitrarily large by presenting two infinite families that attain all possible values for the stated differences. This is in contrast with the trivial fact that $\mu_t(G)\le \mu(G)$ in any graph $G$. We conclude the paper in Section~\ref{sec:conc} with a number of remarks and open problems.

\section{Related problems and Neighborhood Lemma}
\label{sec:related}

Our study has connections with several known topics in graph theory. One of them is a newly introduced topic that comes from the well known, and recently very active, area of general position problems, while another one is a classical combinatorial problem that goes back to Erd\H{o}s, Hajnal and Moon.

We may recall that the general position problem aims to find the cardinality of a largest set $X$ of vertices in a graph $G$ such that no shortest path between a pair of vertices of $X$ contains a third vertex of $X$. (Note that, connecting this definition to that of mutual-visibility sets, one can roughly say that $X$ is a set of vertices of $G$ such that each two vertices of $X$ are $X$-visible through all the possible shortest paths between them.) The \emph{general position number} is then the cardinality of a largest such set in a graph $G$, and is denoted by $\gp(G)$. The concept already appeared in \cite{Korner}, where it was considered in the context of graph theory (concerning the class of hypercubes) for the first time, and was more recently rediscovered in \cite{ullas-2016,manuel-2018}.

Clearly, a general position set is also a mutual-visibility set in any  \emph{connected} graph $G$, and consequently,
\begin{equation}\label{eq:gp-mu}
\mu(G)\ge \gp(G).
\end{equation}

\subsection{Neighborhood Lemma and \bf{some} applications}
\label{ss:neighborhood}

In this subsection, we present a useful tool for proving upper bounds on the lower mutual-visibility number of graphs.
It somehow indicates that appropriate maximal mutual-visibility sets can also be considered locally.

\begin{lem}
{\em [Neighborhood Lemma]}
\label{prp:neighborhood}
Let $G$ be a connected graph and $x\in V(G)$. The set $N[x]$ is a maximal mutual-visibility set if and only if for every two vertices $u,v\in N(x)$ we have $uv\in E(G)$ or there exists $w\in G-N[x]$, which is a common neighbor of $u$ and $v$ in $G$.
\end{lem}
\proof
Let $N[x]$ be a maximal mutual-visibility set in $G$, and let $u,v\in N(x)$ such that $uv\notin E(G)$. Since $u$ and $v$ are $N[x]$-visible, and $d_G(u,v)=2$, there should be a vertex $w$ outside $N[x]$ so that $uwv$ is a (shortest) path. Thus, $w\in N(u)\cap N(v)$, as claimed.

Conversely, let $x$ be a vertex in $G$ such that  for every two vertices $u,v\in N(x)$ we have $uv\in E(G)$ or there exists $w\in G-N[x]$ so that $w\in N(u)\cap N(v)$. The latter condition ensures that every two vertices in $N(x)$ are $N[x]$-visible. Clearly, $x$ is $N[x]$-visible to all its neighbors.  Thus, $N[x]$ is a mutual-visibility set. It is also clear that $N[x]$ is a maximal mutual-visibility set, because for $S=N[x]\cup\{z\}$, where $z$ is any vertex in $V(G)\setminus N[x]$, the vertices $x$ and $z$ are not $S$-visible.
\qed

\bigskip

We follow with some applications of the Neighborhood Lemma.
In the event that there exists a vertex $x\in V(G)$ that admits the conditions in Lemma~\ref{prp:neighborhood}, we get the following upper bound: $$\ml(G)\le |N(x)|+1\le \Delta(G)+1.$$ In particular, the conditions are fulfilled if $x$ is a simplicial vertex, in which case we also get the bound $\ml(G)\le \omega(G)$, where $\omega(G)$ is the cardinality of a largest clique.
We infer the following result:

\begin{cor}
If $G$ is a chordal graph, then $\ml(G)\le \omega(G)$.
\end{cor}

Next, consider a non-trivial cograph $G$, and let $x,y\in V(G)$ be two vertices in $G$ that have the same neighborhoods. Now, if the conditions of Lemma~\ref{prp:neighborhood} hold for $x$, then $N_G[x]$ is a maximal mutual-visibility set, and $\ml(G)\le \Delta(G)+1$. On the other hand, if the conditions of Lemma~\ref{prp:neighborhood} are not fulfilled for $x$, then $X=N_G[x]\setminus \{y\}$ is a mutual-visibility set. Indeed, any two vertices in $N_G(x)\setminus\{y\}$ that are not adjacent are $X$-visible by the shortest path that goes through $y$. Clearly, $X$ is also a maximal mutual-visibility set. These observations yield the following result.
\begin{cor}
If $G$ is a non-trivial cograph, then $\ml(G)\le \Delta(G)+1$.
\end{cor}

Another application of the Neighborhood Lemma is in the class of Cartesian grids. We mention that the graph $P_n\cp P_m$ was earlier studied in some papers concerning mutual-visibility parameters. In particular, it was proved by Di Stefano in~\cite{DiS} that $\mu(P_n\cp P_m)=2\min\{m,n\}$. Now, consider the lower mutual-visibility number of grids. Denoting $V(P_m)=[m]$, we note that the neighbors $(1,2)$ and $(2,1)$ of the vertex $(1,1)$ in the graph $P_m\cp P_n$ are at distance $2$, and there is a path of length two between them, which avoids $(1,1)$. Thus, by Lemma~\ref{prp:neighborhood}, we infer that $S=\{(1,1),(2,1),(1,2)\}$ is a maximal mutual-visibility set in $P_m\cp P_n$. Hence,
$\ml(P_m\cp P_n)\le 3$.
On the other hand, it is easy to see (and it follows directly from Theorem~\ref{thm_cut-edge} which we prove in Section~\ref{sec:bounds}) that $\ml(P_m\cp P_n)>2$.
(See Fig.~\ref{fig:grid}, where a maximal mutual-visibility set of the grid $P_{15}\cp P_{7}$ is shown.)

\begin{cor}
If $m,n\ge 2$, then $\ml(P_m\cp P_n)=3$.
\end{cor}

\begin{figure}[!ht]
\centering
\begin{tikzpicture}[scale=.75]\footnotesize
\def\vr{3.5pt} 

\begin{scope}<+->;

  \draw[step=1cm,gray,very thin] (0,0) grid (14,6);
\path (0,0) coordinate (a);
\path (0,1) coordinate (b);
\path (1,0) coordinate (c);

\draw (a) [fill=black] circle (\vr);
\draw (b) [fill=black] circle (\vr);
\draw (c) [fill=black] circle (\vr);

\draw (0,-0.6) node {1};
\draw (1,-0.6) node {2};
\draw  (2,-0.6) node {3};
\draw  (3,-0.6) node {4};
\draw  (4,-0.6) node {\ldots};
\draw (-0.6,0) node {1};
\draw (-0.6,1) node {2};
\draw (-0.6,2) node {3};
\draw (-0.6,3) node {4};
\draw (-0.6,4) node {\vdots};

\end{scope}

\end{tikzpicture}
\caption{A maximal mutual-visibility set in $P_{15}\Box P_{7}$ depicted with black circles.}
\label{fig:grid}
\end{figure}
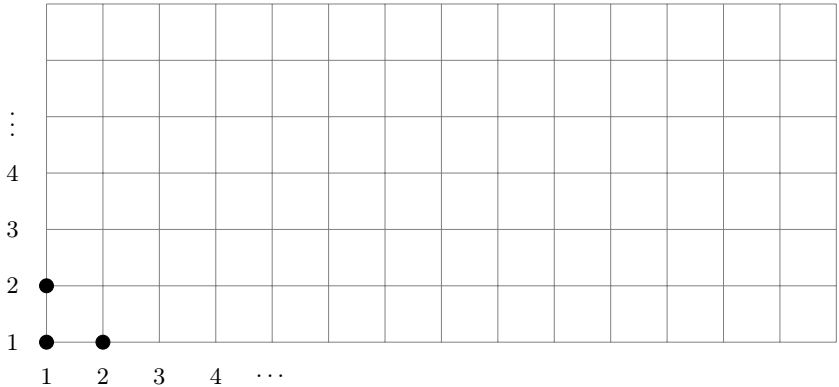

We will use the Neighborhood Lemma in several further results in the paper.

\subsection{Lower version of the general position problem}
\label{ss_generalposition}

In connection with the general position problem, and the smallest possible sets with a maximality property, the lower general position problem has been recently considered in \cite{DiSKKTY} as follows. Given a graph $G$, a set $X\subset V(G)$ is a {\em maximal general position set} in $G$ if $X$ is a general position set (that is, for every two vertices $u,v\in X$, $I_G(u,v)\cap X=\{u,v\}$), and every set $Y$, where $X\subsetneq Y$, is not a general position set. The cardinality of a smallest maximal general position set is the {\em lower general position number} of $G$, denoted by $\gp^-(G)$.

One might think that, in view of the relationship~\eqref{eq:gp-mu} between general position and mutual-visibility numbers, it could be expected that a similar inequality would hold between their lower versions. However, this is far from reality, since the lower parameters are not in general comparable, as we now show. To this end, we consider the following construction.

We begin with an arbitrarily large set of isolated vertices $B$, and three extra vertices $a,a',b$. Then we add the edges $ab, a'b$ and $ax, a'x$ for every $x\in B$. We next add three (arbitrarily large) cliques $K_t$, $K_{t'}$ and $K_{t''}$ and all the edges $ax, bx$ with $x\in V(K_t)$, $bx$ with $x\in V(K_{t'})$, and $a'x, bx$ with $x\in V(K_{t''})$. We denote the graph thus obtained as $G^*$. See Figure \ref{fig:G-prime} for a sketch of such a graph $G^*$.

\begin{figure}[ht]
\centering
\begin{tikzpicture}[scale=.65, transform shape]

\node [draw, shape=circle, scale=.8] (b1) at  (0, 0) {};
\node [draw, shape=circle, scale=.8] (b2) at  (0, 1) {};
\node [draw, shape=circle, scale=.8] (b3) at  (0, 2) {};
\node [draw, shape=circle, scale=.8] (b) at  (0, 3) {};
\node [draw, shape=circle, scale=.8] (a) at  (-5, 2) {};
\node [draw, shape=circle, scale=.8] (a1) at  (5, 2) {};

\node [scale=1.3] at (-5.4,1.7) {$a$};
\node [scale=1.3] at (5.4,1.7) {$a'$};
\node [scale=1.3] at (0,3.55) {$b$};

\node [scale=1.3] at (4.5,4.5) {$K_{t''}$};
\node [scale=1.3] at (-4.5,4.5) {$K_t$};
\node [scale=1.3] at (0,5.5) {$K_{t'}$};
\node [scale=1.3] at (-0.8,-0.1) {$B$};

\draw[rounded corners=2pt] (-0.4,-0.4) rectangle ++(0.8,2.8);

\draw[rotate around={-20:(4.5,4.5)}] (4.5,4.5) ellipse (1.5cm and 0.7cm);
\draw[rotate around={20:(-4.5,4.5)}] (-4.5,4.5) ellipse (1.5cm and 0.7cm);
\draw[] (0,5.5) ellipse (1.5cm and 0.7cm);

\draw(a)--(b1)--(a1)--(b2)--(a)--(b3)--(a1)--(b)--(a);

\draw(a)--(-5.91,3.95);
\draw(a)--(-3.2,4.55);
\draw(a1)--(5.91,3.95);
\draw(a1)--(3.2,4.55);
\draw(b)--(-1.45,5.3);
\draw(b)--(1.45,5.3);
\draw(b)--(-5.5,3.69);
\draw(b)--(5.5,3.69);
\draw(b)--(-3.3,5.23);
\draw(b)--(3.3,5.23);

\end{tikzpicture}
\caption{\small A sketch of a graph $G^*$.}\label{fig:G-prime}
\end{figure}
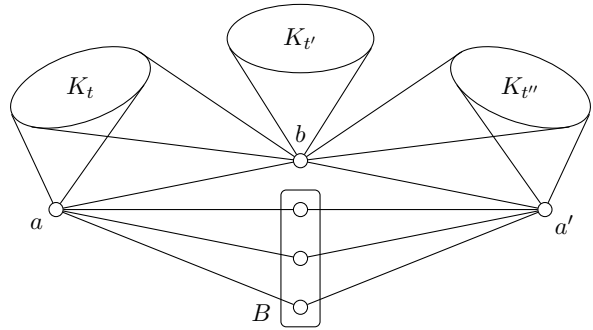

\begin{prop}
\label{prop:comparison-1}
For any graph $G^*$, $\mu^-(G^*)\le 3$ and $\gp^-(G^*)\gg 3$.
\end{prop}

\begin{proof}
We first claim that $X=\{a,a',b\}$ is a maximal mutual-visibility set. Clearly, these three vertices are $X$-visible, and so, $X$ is a mutual-visibility set. Now, observe that no vertex $w\in B$ can be added to $X$ keeping the mutual-visibility property because then $b, w$ would not be visible. Similarly, no vertex from $V(K_t)$, $V(K_{t'})$ or $V(K_{t''})$ can be added to $X$, since they would be not visible with $a$ or with $a'$. Thus, $X$ is maximal, and so $\mu^-(G^*)\le 3$.

Now, to see that $\gp^-(G^*)\gg 3$ consider the following arguments. Let $S$ be a maximal general position set of the smallest cardinality. If $S\cap V(K_{t})\ne \emptyset$, or $S\cap V(K_{t'})\ne \emptyset$, or $S\cap V(K_{t''})\ne \emptyset$, then $V(K_{t})\subseteq S$, or $V(K_{t'})\subseteq S$, or $V(K_{t''})\subseteq S$, respectively, and so, $\gp^-(G^*)\gg 3$, since $K_t$, $K_{t'}$ and $K_{t''}$ are arbitrarily large cliques. Hence, we may assume that $S\cap (V(K_t)\cup V(K_{t'})\cup V(K_{t''}))=\emptyset$. If neither $a$ nor $a'$ are in $S$, then it must be $S=B\cup\{b\}$, and so, $\gp^-(G^*)\gg 3$, since $B$ has an arbitrarily large cardinality (and at least larger than $2$). By symmetry, we may assume that $a\in S$, and consider two cases.

\medskip
\noindent
Case 1: $a'\in S$ (and $a\in S$). In this case, $(B\cup \{b\})\cap S=\emptyset$. But then $S$ is not maximal, since $\{a,a'\}\cup (V(K_t)\cup V(K_{t'})\cup V(K_{t''}))$ is a general position set, a contradiction.

\medskip
\noindent
Case 2: $a'\notin S$ (and $a\in S$). In this situation, $S$ could contain at most one vertex from the set $B\cup \{b\}$. If $b\in S$, then $S$ is not maximal, since $\{a,b\}\cup (V(K_t))$ is a general position set, a contradiction. On the other hand, if $b\notin S$, then again $S$ is not maximal since $\{a,x\}\cup (V(K_{t''}))$ (where $x\in B$) is a general position set.
\end{proof}

From the above proof we can deduce that $\gp^-(G^*)\ge\min\{t,t',t'',|B|\}$. By the definition of the  graphs $G^*$ each of the values on the right side of the equality can be as large as one wants.

The proposition above shows that there are graphs $G$ such that $\mu^-(G)<\gp^-(G)$. Moreover, the examples of complete bipartite graphs show that this inequality can also be reversed.

\begin{prop}
\label{prop:comparison-2}
For any $r\ge s\ge 1$, $\mu^-(K_{r,s})=s+1$ and $\gp^-(K_{r,s})=2$.
\end{prop}

\begin{proof}
If $s=1$, then $K_{r,s}$ is a star and one can easily verify that $\mu^-(K_{r,s})=2=s+1$ and $\gp^-(K_{r,s})=2$. We may next assume that $r \ge s \ge 2$. Let $U$ and $U'$ be the bipartition sets of $K_{r,s}$ of cardinality $r$ and $s$, respectively. The equality $\gp^-(K_{r,s})=2$ was proved in \cite{DiSKKTY}. On the other hand, consider the set $X=U'\cup \{u\}$ where $u\in U$. Note that $X=N[u]$ and the conditions of Lemma~\ref{prp:neighborhood} are fulfilled. Therefore, $X$ is a maximal mutual-visibility set,
which leads to $\mu^-(K_{r,s})\le s+1$.

On the other hand, let $X'$ be a maximal mutual-visibility set of $K_{r,s}$. First observe that $X'$ is neither a subset of $U$ nor of $U'$. Indeed, if, for instance, $X'\subseteq U$, then we can extend $X'$ to a larger mutual-visibility set by adding any vertex of $U'$ (notice that $U'$ has cardinality at least two), which is a contradiction since $X'$ is maximal. Thus, $X'\cap U\ne \emptyset$ and $X'\cap U'\ne \emptyset$. Also, if $U\subset X'$, then $|X'\cap U'|=1$, for otherwise, there are two vertices of $X'\cap U'$, which are not $X'$-visible. Thus, $|X'|=|U|+1\ge s+1$.  An analogous conclusion follows if $U'\subset X'$.

In this sense, we may assume $|X'\cap U|<|U|=r$ and $|X'\cap U'|<|U'|=s$. If $|X'\cap U|<r-1$, then we can extend $X'$ to a larger mutual-visibility set by adding to $X'$ one vertex of $U$ not yet in $X'$, which is not possible. Thus, $|X'\cap U|=r-1$. By similar arguments, we also  deduce that $|X'\cap U'|=s-1$. Altogether, we obtain that $|X'|\ge r+s-2\ge s+1$ (when $r\ge 3$), which gives the desired equality. If $r\in \{1,2\}$, then $K_{r,s}$ is either $P_2$, $P_3$ or $C_4$, where clearly $\gp^-(K_{r,s})=s+1$.
\end{proof}

\subsection{Relation to Bollob\' as-Wessel theorem}
\label{ss:bolobas}

The problem of mutual-visibility in Cartesian products of two complete graphs is intrinsically related to the famous Zarankiewicz problem, which is still open. More directly, it was noticed by Cicerone, Di Stefano and Klav\v{z}ar  in~\cite{CDK} that $\mu(K_m\cp K_n)$ equals $z(m,n;2,2)$, which is the maximum number of 1s in an $m\times n$ binary matrix that contains no constant $2\times 2$ submatrix of 1s;
see~\cite{CDK} for more details. Here we present a similarly strong connection between the lower mutual-visibility number of Cartesian products of two complete graphs with another old result related to binary matrices. In fact, the result can also be presented in terms of complete bipartite subgraphs of bipartite graphs, which was first conjectured by Erd\H{o}s, Hajnal and Moon~\cite{EHM}.

Let $G$ be a bipartite graph with bipartition sets of cardinalities $m$ and $n$. The graph $G$ has the {\em property $(k,\ell)$} if adding any new edge to $G$ increases the number of complete bipartite subgraphs $K_{k,\ell}$ of $G$. It was conjectured in~\cite{EHM} that a bipartite graph with bipartition sets of cardinalities $m$ and $n$ that satisfies the property $(k,\ell)$ has at least $(k-1)m+(\ell-1)n-(k-1)(\ell-1)$ edges. The conjecture was proved independently by Wessel~\cite{wes} and Bollob\'{a}s~\cite{bol}. The special case of the result when $k=\ell=2$ is related to the topic of this paper. Note that  the Bollob\'{a}s-Wessel theorem in this case states that, when $G$ is a bipartite graph with bipartition sets of cardinalities $m$ and $n$, if adding any new edge increases the number of $4$-cycles (that is, subgraphs isomorphic to $C_4$), then $G$ has at least $m+n-1$ edges.

Consider the lower mutual-visibility problem in the Cartesian product $K_m\cp K_n$ of two complete graphs, whose vertex sets are denoted by $[m]$ and $[n]$. Let $T=(\{1\}\times [n])\cup ([m]\times \{1\})$ be the subset of $[m]\times [n]$. Clearly, $T$ is the closed neighborhood of the vertex $(1,1)$, and the conditions of Lemma~\ref{prp:neighborhood} are fulfilled. Therefore,
$T$ is a maximal mutual-visibility set and $\ml(K_m\cp K_n)\le |T|=m+n-1$.

On the other hand, note that the set of vertices of the Cartesian product $K_m\cp K_n$ uniquely corresponds to the edge set of the bipartite graph $B$ with bipartition sets of cardinalities $m$ and $n$ defined as follows. Denoting the bipartition sets of $B$ by $[m]$ and $[n]$, we have $ij\in E(B)$ if and only if $(i,j)\in V(K_m\cp K_n)$ for any $i\in [m]$ and $j\in [n]$.  Now we have that a given set $S\subset V(K_m\cp K_n)$ is a mutual-visibility set of $K_m\cp K_n$ if and only if $S$ does not contain a subgraph isomorphic to $C_4$ (this fact was proved in \cite{CDK}). Equivalently, this means that the subgraph of $B$ induced by the edges in $B$ that correspond to the vertices of $S$ in $K_m\cp K_n$ do not contain $K_{2,2}$ as a subgraph.

Suppose that $S$ is a mutual-visibility set of $K_m\cp K_n$, where $|S|<m+n-1$. We will show that $S$ is not maximal. Consider the edges $S'$ in $B$ corresponding to the vertices of $S$. Clearly $|S'|=|S|<m+n-1$. By the Bollob\'{a}s-Wessel theorem, there exists an edge not in $S'$ (say, $i'j'$), which one can add to $S'$ and the number of subgraphs isomorphic to $K_{2,2}$ in $S'$ does not increase (that is, it remains zero). Translating this to the corresponding mutual-visibility set $S$ in $K_m\cp K_n$ we observe that $S$ is not a maximal mutual-visibility set, because $S\cup\{(i',j')\}$ is a mutual-visibility set in $K_m\cp K_n$. We infer the following result.

\begin{cor}
\label{cor:K-n-K-m}
If $m,n$ are positive integers, then $\ml(K_m\cp K_n)=m+n-1$.
\end{cor}

\section{Computational complexity}
\label{sec:complexity}

This section is focused on computational aspects of the lower total mutual-visibility number of graphs. That is, we consider the following decision problem:

\begin{center}
\fbox{
	\parbox{\textwidth}{
		\textsc{Lower Total Mutual-Visibility Problem} \\
		\textit{Input}: A connected graph $G = (V,E)$ and $k\le n(G)$.\\
		\textit{Question}: Is $\mu^-_t(G)\leq k$?}}

\end{center}

The fact that the \textsc{Lower Total Mutual-Visibility Problem} belongs to the class NP is, unlike with many other similar problems, not obvious. For this purpose we will need to use the following remark, which follows from the fact that if a set of vertices $X$ is not a total mutual-visibility set, then no superset of $X$ is a total mutual-visibility set.

\begin{remark}
\label{rem:cond-maximal}
Let $G$ be a graph and let $X$ be a total mutual-visibility set of $G$. Then, $X$ is maximal if and only if $X\cup \{w\}$ is not a total mutual-visibility set for every $w\in V(G)\setminus X$.
\end{remark}

To prove the NP-completeness of the problem, we present a polynomial reduction from the \textsc{Independent dominating set problem} to the \textsc{Lower Total Mutual-Visibility Problem}. The former problem was already shown to be NP-complete in the book of Garey and Johnson~\cite{GJ}. For the reduction, we follow a construction already given in~\cite{CDDH}. We may also recall that an \emph{independent dominating set} $S$ of a graph $G$ is a set of vertices that is independent and all vertices in $V(G)-S$ have a neighbor in $S$.

The construction is made as follows. Let $G$ be a graph with $V(G) = [n]$. For every edge $e=ij$ of $G$, a vertex $v_e=v_{ij}$ is added as well as the edges $iv_e$ and $jv_e$. Also, all possible edges between all the vertices $v_e$, where $e\in E(G)$, are added so that these vertices induce a clique $K_m$. Let $t\ge 3$ be an integer. A clique $K_{t+1}$ is added and one of its vertices, denoted by $x$, is chosen, so that each vertex of $G$ is joined by an edge to $x$. Assume that $V(K_{t+1})=\{x,x_1,\dots,x_t\}$. Note that $\{x_1,\dots,x_t\}$ induce a clique $K_t$. In addition, for each vertex $v_e$ with $e\in E(G)$, a clique $K_t$ with vertex set $V(K_t)=\{e_{y_1},\dots,e_{y_t}\}$ is added, and each vertex of such $K_t$ is joined by an edge with the corresponding $v_e$. The resulting graph is denoted by $G'$. A drawing of a fairly representative example of the graph $G'$, when $G = P_5$, was given in~\cite[Figure 1]{CDDH}. However, in order to facilitate the reading, we next include a similar drawing when $G$ is the star $S_4$ on four leaves.

\begin{figure}[!ht]
\centering
\begin{tikzpicture}[scale=1.5, style=thick]
\def\vr{3pt}
\def\len{1}

\coordinate(G) at (-1,0);
\coordinate(a) at (0,0);
\coordinate(b) at (1,0);
\coordinate(c) at (2,0);
\coordinate(d) at (3,0);
\coordinate(e) at (4,0);
\coordinate(x) at (2,1);
\coordinate(x1) at (0.5,-1);
\coordinate(x2) at (1.5,-1);
\coordinate(x3) at (2.5,-1);
\coordinate(x4) at (3.5,-1);

\coordinate(k1) at (2.4,2);
\coordinate(k2) at (0.7,-2.2);
\coordinate(k3) at (1.7,-2.2);
\coordinate(k4) at (2.7,-2.2);
\coordinate(k5) at (3.7,-2.2);
\coordinate(k6) at (4.2,-1);

\coordinate(l1) at (1.6,1.9);
\coordinate(d1) at (2.4,1.9);
\coordinate(l2) at (0.1,-1.9);
\coordinate(d2) at (0.9,-1.9);
\coordinate(l3) at (1.1,-1.9);
\coordinate(d3) at (1.9,-1.9);
\coordinate(l4) at (2.1,-1.9);
\coordinate(d4) at (2.9,-1.9);
\coordinate(l5) at (3.1,-1.9);
\coordinate(d5) at (3.9,-1.9);

\draw (b)--(c)--(d);
\draw (a) .. controls (0.2,0.3) and (1.8,0.3) .. (c) .. controls (2.2,0.3) and (3.8,0.3) .. (e);

\draw (a)--(x);
\draw (b)--(x);
\draw (c)--(x);
\draw (d)--(x);
\draw (e)--(x);

\draw (a)--(x1);
\draw (c)--(x1);
\draw (b)--(x2);
\draw (c)--(x2);
\draw (c)--(x3);
\draw (d)--(x3);
\draw (c)--(x4);
\draw (e)--(x4);

\draw (x)--(l1);
\draw (x)--(d1);
\draw (x1)--(l2);
\draw (x1)--(d2);
\draw (x2)--(l3);
\draw (x2)--(d3);
\draw (x3)--(l4);
\draw (x3)--(d4);
\draw (x4)--(l5);
\draw (x4)--(d5);

\draw(a)[fill=white] circle(\vr);
\draw(b)[fill=white] circle(\vr);
\draw(c)[fill=white] circle(\vr);
\draw(d)[fill=white] circle(\vr);
\draw(e)[fill=white] circle(\vr);
\draw(x)[fill=white] circle(\vr);
\draw(x1)[fill=white] circle(\vr);
\draw(x2)[fill=white] circle(\vr);
\draw(x3)[fill=white] circle(\vr);
\draw(x4)[fill=white] circle(\vr);

\draw (0.5,-2) ellipse (0.4cm and 0.25cm);
\draw (1.5,-2) ellipse (0.4cm and 0.25cm);
\draw (2.5,-2) ellipse (0.4cm and 0.25cm);
\draw (3.5,-2) ellipse (0.4cm and 0.25cm);
\draw (2,2) ellipse (0.4cm and 0.25cm);
\draw (2.1,-1) ellipse (2.1cm and 0.3cm);

\draw[anchor = west] (x)++(0.05,0.05) node {$x$};
\draw[anchor = west] (x1)++(0.05,0.0) node {$v_{e_1}$};
\draw[anchor = west] (x2)++(0.05,0.0) node {$v_{e_2}$};
\draw[anchor = west] (x3)++(0.05,0.0) node {$v_{e_3}$};
\draw[anchor = west] (x4)++(0.05,0.0) node {$v_{e_4}$};
\draw[anchor = west] (G) node {$S_4:$};

\draw[anchor = west] (k1) node {$K_t$};
\draw[anchor = west] (k2)++(0.05,0.0) node {$K_t$};
\draw[anchor = west] (k3)++(0.05,0.0) node {$K_t$};
\draw[anchor = west] (k4)++(0.05,0.0) node {$K_t$};
\draw[anchor = west] (k5)++(0.05,0.0) node {$K_t$};
\draw[anchor = west] (k6) node {$K_m$};

\end{tikzpicture}
\caption{The graph $G'$ from the star $G=S_4$ on four leaves.}
\label{fig:G'}
\end{figure}
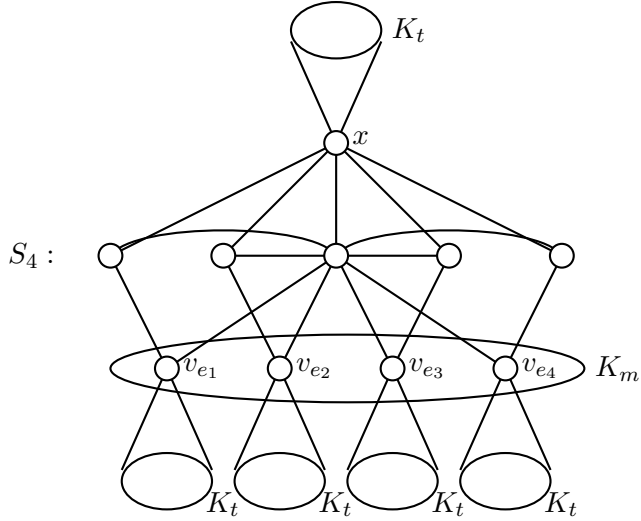

\begin{thm}
\textsc{Lower Total Mutual-Visibility Problem} is NP-complete.
\end{thm}

\begin{proof}
In the proof, we use similar arguments as the ones used in~\cite[Theorem 3.1]{CDDH} to prove some related complexity results. However, instead of using a reduction from the \textsc{Independent Set Problem}, in this proof, we use a reduction from the \textsc{Independent dominating set problem}.

We first observe that the \textsc{Lower Total Mutual-Visibility Problem} is in NP, since one can check in polynomial time whether a given set is indeed a total mutual-visibility set, and also, by using Remark \ref{rem:cond-maximal}, that it is maximal.

Let $G$ be an arbitrary connected graph and consider the construction $G'$ from $G$ as described above. Let $m=|E(G)|$, and let $X\subset V(G')$ contain all the vertices of all the $m+1$ involved cliques $K_t$ whose vertices are simplicial vertices together with the vertices of an independent dominating set $I$ of $G$. We claim that $X$ is a maximal total mutual-visibility set of $G'$.

Indeed, by using the arguments of the proof of~\cite[Theorem 3.1]{CDDH}, we derive that $X$ is a total mutual-visibility set of $G'$. Now, observe that none of the vertices from  the set $\{x\}\cup V(K_m)$ (which are cut-vertices of $G'$) can be added to $X$ keeping the total mutual-visibility property for $X$ in $G'$. Moreover, if we add a vertex $i\in V(G)\setminus I$ to $X$, then since $I$ is an independent dominating set of $G$, there is $j\in I$ such that $e=ij\in E(G)$. This means that the vertices of the clique $K_t$ adjacent to the vertex $v_e\in V(K_m)$ are not visible with any other vertex from the clique $K_{t}$ whose vertices are adjacent to the vertex $x$. This is a contradiction, which implies that $X$ is a maximal total mutual-visibility set, as desired. Thus, $\mu_t^-(G')\le t(m+1)+i(G)$.

On the other hand, let $X'$ be a maximal total mutual-visibility set in $G'$ of the smallest cardinality. Notice that none of the vertices of the set $\{x\}\cup V(K_m)$ can be in $X'$, since they are cut-vertices of $G'$. Also, all the vertices from all the $m+1$ involved cliques $K_t$ whose vertices are simplicial vertices must be in $X'$ as well, since they are simplicial vertices in $G'$. Now, if $I'=X'\cap V(G)$ is not an independent set, then there is an edge $e=ij\in E(G)$ such that $i,j\in I'$. Thus, similarly to a previous comment, we will find vertices that are not $X'$-visible, which is not possible. Thus, $I'$ must be independent. In addition, if $I'$ is not a maximal independent set of $G$, then $X'$ is not a maximal total mutual-visibility set, since it can be extended to a larger total mutual-visibility set by adding more vertices from $G$ being not adjacent to any vertex of $I'$, which is not possible. Thus, $I'$ is a maximal independent set of $G$. Consequently, we deduce that $\mu_t^-(G')=|X'|\ge |I'|+t(m+1)\ge i(G)+t(m+1)$. Therefore, we deduce that $\mu_t^-(G')=t(m+1)+i(G)$, by which the reduction from the \textsc{Independent dominating set problem} is complete. It is also easy to see that one can construct $G'$ from $G$ in polynomial time.
\end{proof}

The reduction above (as well as the ones presented in~\cite[Theorem 3.1]{CDDH}) cannot be directly adapted to prove a similar conclusion to the above one, for the case of the lower mutual-visibility problem. This problem seems to be more challenging, and we leave it as an open question.

\section{General bounds and extremal families}
\label{sec:bounds}
We start this section with general bounds on the (total) mutual-visibility number.

\begin{prop}
\label{prp:basic}
If $G$ is a connected graph, then
\begin{enumerate}[{\rm (i)}]
\item $1\le  \ml(G)\le \mu(G)$, and $\ml(G)=1$ if and only if $G=K_1$;

\item $0\le \tml(G)\le \mu_t(G)$ and $\tml(G)=0$ if and only if $\mu_t(G)=0$.
\end{enumerate}
\end {prop}
\proof Note that the inequalities in (i) are trivial. To see the second statement of (i), note that $K_1$ is the only (connected) graph with $\ml(K_1)=1$. Indeed, if $G$ is a graph with an edge $e=uv$, then it is clear that $\{u,v\}$ is a mutual-visibility set, which implies $\ml(G)\ge2$.

Concerning (ii), it is again trivial that $0\le \tml(G)\le \mu_t(G)$ and that $\mu_t(G)=0$ implies $\tml(G)=0$. Now, let $\tml(G)=0$, and suppose, to the contrary, that $\mu_t(G)>0$. Then, there exists a set $X\subset V(G)$, which is a total mutual-visibility set of cardinality $|X|>0$. Hence, the empty set, $\emptyset$, is not a maximal mutual-visibility set, and so $\tml(G)>0$, a contradiction.
\qed

\bigskip

It is easy to see that a vertex, which is the center of a convex $P_3$, does not lie in a total mutual-visibility set (because the ends of the convex $P_3$ would in this case not be visible). Tian and Klav\v{z}ar~\cite{TK} extended this observation by giving a characterization of the graphs $G$ with $\mu_t(G)=0$, which in view of Proposition~\ref{prp:basic} yields the following result.

\begin{cor}
Let $G$ be a graph. Then $\tml(G)=0$ if and only if every vertex is the central vertex of a convex $P_3$ in $G$.
\end{cor}

We continue with an auxiliary result, which can also be of independent interest.

\begin{lem}
\label{lem:cut-edge}
Let $e=uv$ be an edge in a connected graph $G$. Then $\{u,v\}$ is a maximal mutual-visibility set of $G$ if and only if $e$ is a cut-edge in $G$.
\end{lem}

\proof
Clearly, $\{u,v\}$ is a mutual-visibility set. Assume first $\{u,v\}$ is a maximal mutual-visibility set, and suppose that $e$ is not a cut-edge. Then $e=uv$ lies on a cycle, and let $u=u_0u_1\cdots u_k=vu$, where $k\ge 2$, be a shortest cycle on which $e$ lies. Note that $\{u,u_{\lfloor\frac{k}{2}\rfloor},v\}$ is a mutual-visibility set, which is a contradiction to the maximality of $\{u,v\}$ as a mutual-visibility set. Conversely, if $e=uv$ is a cut-edge in $G$, then for any vertex $w\in V(G)\setminus\{u,v\}$  either $u$ lies on every (shortest) $v,w$-path in $G$ or $v$ lies on every (shortest) $u,w$-path in $G$. This implies that $\{u,v\}$ is a maximal mutual-visibility set.
\qed
\bigskip

Using Lemma~\ref{lem:cut-edge}, we next characterize the graphs with the lower mutual-visibility number equal to $2$.

\begin{thm}
\label{thm_cut-edge}
Let $G$ be a connected graph on at least two vertices. Then, $\ml(G)=2$ if and only if $G$ has a cut-edge.
\end{thm}

\proof
If $G$ has a cut-edge $e=uv$, then, by Lemma~\ref{lem:cut-edge}, $\{u,v\}$ is a maximal mutual-visibility set. Hence $\ml(G)\le 2$, and, by Proposition~\ref{prp:basic}~(i), we get $\ml(G)=2$.

Conversely, let $\ml(G)=2$, and let $\{u,v\}$ be a maximal mutual-visibility set in $G$. We may assume that $uv\notin E(G)$, for otherwise the proof is done by Lemma~\ref{lem:cut-edge}. Hence, $d_G(u,v)=k\ge 2$. First, we claim that $I_G(u,v)$ consists only of vertices in exactly one (shortest) $u,v$-path. Let $P:u=u_0 \cdots u_k=v$ be a shortest $u,v$-path. Suppose now that $I_G(u,v)\ne V(P)$, and let $w$ be a vertex closest to $u$ that is in $I_G(u,v)\setminus V(P)$. Then, $\{u,w,v\}$ is a mutual-visibility set. Indeed, the shortest $u,v$-path $P$ avoids $w$, and since $w\in I_G(u,v)$, there is a shortest $u,w$-path avoiding $v$ and a shortest $w,v$-path avoiding $u$. This is a contradiction with the assumption that $\{u,v\}$ is a maximal mutual-visibility set.

Next, we claim that there is no vertex in $V(G)\setminus V(P)$ that is adjacent to a vertex in $\{u_1,\ldots,u_{k-1}\}$. Suppose that $x\in V(G)\setminus V(P)$ is adjacent to $u_i$, for some $i\in [k-1]$.  Note that $i\le d_G(u,x)\le i+1\le k$, and $k-i\le d_G(x,v)\le k-i+1\le k$. This implies that $v$ does not lie on a shortest $u,x$-path and $u$ does not lie on a shortest $x,v$-path. In addition, $x\notin I_G(u,v)$, thus the set $\{u,x,v\}$ forms a mutual-visibility set, a contradiction.

Finally, we claim that $uu_1$ is a cut-edge in $G$ (in fact, by similar arguments one can prove that every edge of $P$ is a cut-edge in $G$). Suppose that $uu_1$ is not a cut-edge. Due to the observation in the previous paragraph, every $u,u_1$-path that is not just the path $uu_1$ crosses $v$. Let $Q$ be a shortest $u,v$-path in the subgraph of $G$ induced by $V(G)\setminus \{u_1,\ldots,u_{k-1}\}$. Then $Q$ is an induced path also in $G$, and there is a vertex $w$ in $Q$ such that $|d_G(u,w)-d_G(w,v)|\le 1$. Hence, a shortest $u,w$-path in $G$ is a subpath of $Q$, and a shortest $w,v$-path in $G$ is the complementary subpath of $Q$. In addition, since $w\notin V(P)$, we infer that $\{u,w,v\}$ is a mutual-visibility set in $G$, a contradiction with maximality of $\{u,v\}$. Hence, $uu_1$ is a cut-edge.
\qed
\bigskip

In connection with the result above and the lower general position number previously defined, it was proved in \cite{DiSKKTY} that a graph $G$ satisfies $\gp^-(G)=2$ if and only if $G$ has a universal line. For a metric space (or a graph) $M = (X, d_M)$ and two elements $x,y$ of $X$, a {\em line} ${\cal L}_M(x,y)$ induced by $x,y$ is the set of elements of $X$ given as follows:
$$\{w\in X:\ d_M(x,y) = d_M(x,w) + d_M(w,y)\ {\rm or}\ d_M(x,y) = |d_M(x,w) - d_M(w,y)|\}\,.$$
In this sense, a line is called {\em universal} if it contains the whole set $X$. The class of graphs that have a universal line is not yet known. Indeed, Chen and Chv\'atal~\cite{ChenChv} conjectured that if the number of lines in a metric space is smaller than $|X|$, then $M$ has a universal line, and this question remains open; see~\cite{chvatal-2018,rodriguez-2022} for more on this problem.

Related to Theorem \ref{thm_cut-edge}, since any graph having a cut-edge has a universal line, we deduce that if $\ml(G)=2$ for some graph $G$, then $G$ has a universal line, and so, clearly $\gp^-(G)=2$. The opposite of this is not true, since there are several graphs having a universal line and having no cut-edge  (grid graphs for example), and so, $\ml(G)>2$ in view of Theorem \ref{thm_cut-edge}.

\bigskip

\section{Relationships between the two parameters}
\label{sec:two-parameters}

It is a direct consequence of the definitions that $\mu_t(G)\le \mu(G)$ holds for every graph $G$; see~\cite{CDKY}. However, this is not the case for the lower variants of mutual-visibility, which can be demonstrated by the class of block graphs whose definition we now recall.

A {\em block} in a graph $G$ is a maximal subgraph in $G$ having no cut-vertex. A graph $G$ is a {\em block graph} if each of its blocks is a complete graph. The block graphs are precisely the diamond-free chordal graphs~\cite{BM}.


\begin{thm}
If $G$ is a connected block graph on at least two vertices, $S$ is the set of its simplicial vertices, and $W$ is a maximal clique of minimum cardinality in $G$, then
\begin{enumerate}[{\rm (i)}]
\item $\tml(G)=|S|$, and
\item $\ml(G)=|V(W)|$.
\end{enumerate}
\label{thm:block}
 \end{thm}
\proof
To prove (i) note that any vertex in $V(G)\setminus S$ is the center of a convex $P_3$, hence it cannot be in a total mutual-visibility set. Therefore any total mutual-visibility set is a subset of $S$. In addition, any subset of $S$ is clearly a total mutual-visibility set in $G$. We readily infer that $S$ is the unique maximal total mutual-visibility set in $G$, thus $\tml(G)=|S|$.

For the proof of (ii), first note that vertices of a clique $W$ form a maximal mutual-visibility set in $G$, and so $\ml(G)\le |V(W)|$. Now, let $T$ be a maximal mutual-visibility set in $G$, and for the purpose of getting a contradiction assume that $|T|<|V(W)|$. Since $G$ is a block graph, there exists a clique $W'$ in $G$ such that all vertices of $T$ lie in different components of $G-E(W')$. Since $|V(W')|\ge |V(W)|>|T|$, we infer that there is a component $C$ of $G-E(W')$ having no vertices of $T$. Let $x\in V(W')$ be the vertex that lies in $C$. Now, it is easy to see $T\cup\{x\}$ is a mutual-visibility set of $G$, a contradiction with maximality of $T$. This gives $\ml(G)=|V(W)|$.
\qed

\bigskip

Theorem~\ref{thm:block} implies that the difference $\tml(T)-\ml(T)$ can be arbitrarily large.  In particular, if $G$ is a tree, then $\tml(G)$ equals the number of its leaves, while $\ml(G)=2$.
To see that the  reversed inequality is also possible, consider the graphs $S(K_n)$, which are obtained from the complete graphs $K_n$ by subdividing each of its edges exactly once.
We say that a vertex of $S(K_n)$ is an {\em original vertex} if
it corresponds to a vertex of $K_n$, and it is a {\em subdivided vertex} if it is the result of the subdivision of an edge of $G$.

\begin{thm} If $n\ge 3$, then
\begin{enumerate}[{\rm (i)}]
\item $\tml(S(K_n))=0$, and
\item $\ml(S(K_n))=n$.
\end{enumerate}
\label{thm:subdiv-complete}
\end{thm}
\proof
For the proof of (i), note that $n\ge 3$ implies that $\delta(S(K_n))=2$, and the girth (i.e., the length of a shortest cycle) of $S(K_n)$) is $6$. By the result of Tian and Klav\v{z}ar~\cite[Corollary 3.4]{TK}, which follows from the characterization of the graphs with $\mu_t(G)=0$, every graph with girth at least $5$ and minimum degree at least $2$ has total mutual-visibility number $0$. Hence, we get $\tml(S(K_n))=0$.

Next, we concentrate on (ii). Let $T$ be a maximal mutual-visibility set of $S(K_n)$, and let $G=S(K_n)$. If all vertices in $T$ are original, then due to the maximality of $T$, it is only possible that $T$ consists of all original vertices, thus $|T|=n$. This readily implies that $\tml(G)\le n$. Now, suppose that $T$ contains also some subdivided vertices. We claim that $|T|\ge n$, which suffices for the proof of the theorem.

Let $O$ be the set of original vertices in $T$, and let $S$ the set of subdivided vertices in $T$. Next, let $S_0$, $S_0\subseteq S$, be the set of subdivided vertices in $T$ with a neighbor in $O$. (It is possible that $O$ is empty or that $O\ne \emptyset$ and $S_0=\emptyset$. In these cases the proof simplifies, and we will return to these cases in a later stage of the proof.) Clearly, if $x\in S_0$, then only one of the neighbors of $x$ is in $O$, for otherwise the two neighbors of $x$ from $O$ would not be $T$-visible, a contradiction. From a similar reason, each vertex in $O$ can have at most one neighbor in $S$. Moreover, if $x$ and $y$ are vertices in $S_0$, we claim that $x$ and $y$ have no common neighbor. Suppose that $u$ is a common neighbor of $x$ and $y$. Clearly, $u$ is an original vertex, but $u\notin O$, as noted above. Now, let $v'\in O$ be the other neighbor of $x$ and $v''\in O$ be the other neighbor of $y$, different from $u$. Note that $d_G(v',y)=3$, and there are only two shortest $v',y$-paths, one crossing $x\in T$, and the other crossing $v''\in T$. Hence, $y'$ and $z$ are not $T$-visible, a contradiction. We summarize this part by noting that each vertex in $O$ has at most one neighbor in $S_0$ and no two vertices in $S_0$ have a common neighbor.

Let $O_1$ be the set of original vertices that are not in $O$ and have a neighbor in $S_0$. Clearly, by the above, $|O_1|=|S_0|$.
Now, it is possible that there are no subdivided vertices in $S-S_0$ that have exactly one neighbor in $O_1$. On the other hand, if there are such vertices, let us denote the set of vertices in $S-S_0$ that have exactly one neighbor in $O_1$ by $S_1$, and let $O_2$ be the set of the neighbors of vertices in $S_1$ that do not belong to $O_1$. Since a vertex from $O_2$ is adjacent to at least one (subdivided) vertex from $S_2$, we infer $|O_2|\le |S_1|$.

 Continuing in this way, and assuming that the sets $O_1,\ldots, O_k$, and $S_0,\ldots, S_{k-1}$ have already been determined, we consider the following two possibilities. Either there are no subdivided vertices in $S-(S_0\cup\cdots\cup S_{k-1})$ that have exactly one neighbor in $O_k$ or there are such subdivided vertices. In the former case, the process is finished, and we denote the set of original vertices that do not belong to $O\cup O_1\cup\cdots\cup O_k$ by $O'$. Otherwise, the process continues, and we denote by $S_k$ the set of vertices in $S-(S_0\cup\cdots\cup S_{k-1})$ that have exactly one neighbor in $O_k$, and denote by $O_{k+1}$ the set of the neighbors of vertices in $S_k$ that do not belong to $O_k$. Since a vertex from $O_{k+1}$ is adjacent to at least one (subdivided) vertex from $S_k$, we infer $|O_{k+1}|\le |S_k|$.

 Since the graph is finite, the process eventually ends, with the sets $S_\ell$ and $O_{\ell +1}$, while the set $O'$ can be empty or not. If $O'=\emptyset$, then note that $$|T|=|O|+|S|\ge |O|+|S_0|+\cdots+ |S_\ell|\ge |O|+|O_1|+\cdots+ |O_{\ell+1}|=n,$$
and we are done. So, we are left to consider the case when $|O'|>0$ (Note that in the case when $O=\emptyset$, the set $O'$ consists of all original vertices. In addition, when $O\ne\emptyset$ and $S_0=\emptyset$, then $O'$ consists of all original vertices that are not in $O$.)

First, let $|O'|\ge 3$.
Consider the graph $H$ with $V(H)=O'$ and for two vertices $x,y\in O'$ we have $xy\in E(H)$ if and only if the subdivided vertex $u\in V(G)$, which is adjacent to $x$ and $y$, belongs to $S$. We claim that $\delta(H)\ge 2$. Suppose to the contrary that there is a vertex $x\in O'$ such that $\deg_H(x)\le 1$. Hence, $x$ has at most one common neighbor from $S$ with a vertex in $O'$. Since $|O'|\ge 3$, $x$ has another common neighbor $v$ with a vertex $y$ in $O'$ such that $v$ does not belong to $S$. We claim that $T'=T\cup \{v\}$ is a mutual-visibility set, which will be a contradiction with the maximality assumption on $T$. Clearly, $v$ is $T'$-visible with any vertex $u$ in $S$, which is adjacent to $x$ or to $y$, since $d_H(u,v)=2$, and their common neighbor is not in $T$. It is also easy to see that $v$ is $T'$-visible with a vertex $w$ in $O$. Indeed, $d_H(w,v)=3$, and consider the path from $v$ to $w$ through $x$ and the common neighbor of $x$ and $w$, both of which are not in $T$. It remains to check that $v$ is $T'$-visible to other vertices in $S$ that are at distance $4$ from $v$.  Suppose that $s$ is a vertex in $S$ that has at least one end-vertex $z$ in $O_1\cup \cdots\cup O_{\ell+1}$. Then the path $vxrzs$, where $r$ is the common neighbor of $x$ and $z$, has only its end-vertices in $T'$. We infer that $v$ and $s$ are $T'$-visible. The final possibility is that $s\in S$ is at distance $4$ from $v$ and both neighbors $g$ and $h$ of $s$ are in $O'$. Since $\deg_H(x)\le 1$, at least one of the vertices $g$ or $h$ has the common neighbor with $x$ that is not in $S$. From this we can readily find the path from $v$ to $s$ with no internal vertices from $T'$. To finish the proof that $T'$ is a mutual-visibility set, we need to prove that vertices in $T$ are $T'$-visible. For this purpose we need to consider only those pairs of vertices in $T$, which have $v$ on their shortest path. In addition, this implies that $x$ and $y$ are also on their shortest path. Moreover, this is only possible if $\deg_H(x)=1$, $v'\in S$ is a neighbor of $x$, $v''\in S$ is a neighbor of $y$, and $v'xvyv''$ is a shortest $v'v''$-path. Let $t$ be the other neighbor of $v''$, different from $y$. Since $\deg_H(x)=1$, the common neighbor $u'$ of $x$ and $t$ is not in $S$, and so $v'xu'tv''$ is a shortest $v'v''$-path whose internal vertices are not in $T'$. This yields that $T'$ is indeed a mutual-visibility set, a contradiction, by which the claim that $\delta(H)\ge 2$ is proved. Let $S'$ be the set of vertices in $S$ that are adjacent to two vertices in $O'$. Note that each edge of $H$ uniquely corresponds to a vertex in $S'$. Thus, since $\delta(H)\ge 2$, we infer $|S'|=|E(H)|\ge |V(H)|=|O'|$. Hence,
$$|T|=|O|+|S|\ge |O|+|S_0|+\cdots+ |S_\ell|+|S'|\ge |O|+|O_1|+\cdots+ |O_{\ell+1}|+|O'|=n,$$
as desired.

Finally, let $|O'|\in \{1,2\}$. In this case $O\ne\emptyset$ and $S_0\ne \emptyset$ (implying also $O_1\ne \emptyset$), since $n\ge 3$. Indeed, if $O=\emptyset$, then $|T|\le 1$, which is clearly impossible. On the other hand, if $S_0\ne \emptyset$, then $|S|\le 1$, which leads to a contradiction with maximality of the mutual-visibility set $T$. We distinguish two cases. First, suppose  $|O'|=1$, and let $O'=\{x\}$. Note that $x$ is not adjacent to any vertex in $T$. Letting $T'=T\cup \{x\}$, we get a mutual-visibility set, a contradiction. Second, let $|O'|=2$, and let $O'=\{x,y\}$. We can also assume that the common neighbor $v$ of $x$ and $y$ is in $S$, for otherwise we can again make a larger mutual-visibility set by adding $x$ to $T$, which is the same contradiction as in the case $|O'|=1$. Now, let $z$ be a neighbor of $x$, different from $v$, which is also a neighbor of a vertex in $O_{\ell+1}$. Letting $T'=T\cup\{z\}$, we derive that $T'$ is a mutual-visibility set by following similar lines as in the previous paragraph. By this final contradiction the proof is complete.
\qed

\bigskip

Combining Theorems~\ref{thm:block} and~\ref{thm:subdiv-complete} we infer the following result.

\begin{cor}
For any $k\in\mathbb{N}$ there exist graphs $G$ and $H$ such that $\ml(G)-\tml(G)=k$ and $\tml(H)-\ml(H)=k$.
\end{cor}

\section{Concluding remarks}
\label{sec:conc}

In earlier studies of (total) mutual-visibility problems authors considered grids, Hamming graphs, and other types of graph products. As mentioned in Section~\ref{sec:related}, the mutual-visibility number of Cartesian products of two complete graphs is in close relationship with the well known Zarankiewicz problem. In Section~\ref{ss:bolobas}, we established a similarly close connection between the lower mutual-visibility number of Cartesian products of two complete graphs with a Bollob\'{a}s-Wessel theorem, which enabled us to determine the value of $\ml(K_n\cp K_m)$. The natural question is if one can determine the total mutual-visibility number of Cartesian products of two complete graphs. For the standard version of the total mutual-visibility number, the following formula holds (cf.~\cite[Proposition 4.3]{TK}):
$$\mu_t(K_n\Box K_m)=\max\{m,n\},$$ and we could prove a similar result for the lower total mutual-visibility number (which we state without a proof):

\begin{prop}
\label{prp:product-cliques}
For any $m,n\ge 3$, we have $\tml(K_m\cp K_n)=\min\{m,n\}$.
\end{prop}

In Section~\ref{ss:neighborhood}, we established, as one of the applications of the Neighborhood Lemma, that $\ml(P_m\cp P_n)=3$. Again, we can ask about the lower total mutual-visibility number in grids.
The authors of~\cite{CDDH} observed, based on the results of Tian and Klav\v{z}ar~\cite{TK}, that $\mu_t(P_n\cp P_m)=4$. In fact, they noticed that the total mutual-visibility number of the Cartesian product of $k$ paths, each on at least three vertices, equals $2^k$.
Now, let $S$ be a total mutual-visibility set of $P_{n_1}\cp\cdots\cp P_{n_k}$, where $k\ge 2$ and all $n_i\ge 3$.  If $v$ is a vertex of degree at least $3$, then $v$ is the center of a convex $P_3$, hence $v\notin S$. Hence, only vertices of degree $2$ can lie in $S$. In addition, if not all $2^k$ vertices of degree $2$ are in $S$, then $S$ is not a maximal total mutual-visibility set of $P_{n_1}\cp\cdots\cp P_{n_k}$. We derive the following result.

\begin{prop}
\label{prp:total-grids}
If $k\ge 2$ and $n_i\ge 3$ for all $i\in [k]$, then $\tml(P_{n_1}\cp\cdots\cp P_{n_k})=2^k$.
\end{prop}

As already mentioned, total mutual-visibility appeared naturally in the study of mutual-visibility in strong products of graphs~\cite{CDKY}, while it was further and separately studied for the Cartesian and other products of graphs in~\cite{CDK,Kuziak,TK}.
From this perspective, it would be interesting to consider lower (total) mutual-visibility in Cartesian product of graphs and other graph products. Whilst some general results concerning the (total) mutual-visibility numbers in Cartesian products of graphs would be desirable, one can also restrict to some basic classes of graph products that were considered in earlier papers on mutual-visibility. We explicitly state some of these problems as follows.

\begin{prob}
Determine the lower (total) mutual-visibility numbers of Cartesian cylinders (graphs $C_m\cp  P_n$), tori (graphs $C_m\cp C_n$), and hypercubes $Q_n$.
\end{prob}

In Section~\ref{sec:complexity}, we proved that the decision version of determining the lower total mutual-visibility number is NP-complete. However, we could not establish this for the lower mutual-visibility number, which we now formulate as the following problem.

\begin{prob}
What is the computational complexity of the decision version of determining $\ml(G)$?
 \end{prob}

In Section~\ref{sec:bounds}, we established general bounds on the two newly introduced invariants. For the lower bounds we characterized extremal graphs, which we even extended to graphs $G$ with $\ml(G)=2$; see Theorem~\ref{thm_cut-edge}. In this sense, it would be interesting to characterize the graphs $G$ with $\tml(G)=1$.

Concerning the upper bounds in Proposition~\ref{prp:basic}, we have not explicitly mentioned that they can be attained. Yet, from Proposition~\ref{prp:basic} and~Theorem~\ref{thm_cut-edge} we immediately infer that $\ml(G)=\mu(G)$ for every connected graph $G$ with a cut-edge. Similarly, Proposition~\ref{prp:total-grids} and~\cite[Corollary 4.2]{CDDH} imply that $\tml(P_{n_1}\cp\cdots\cp P_{n_k})=\mu_t(P_{n_1}\cp\cdots\cp P_{n_k})$. In this vein, we conclude the paper with the following problem.

\begin{prob}
Characterize the graphs $G$ with $\ml(G)=\mu(G)$, and $\tml(G)=\mu_t(G)$, respectively.
\end{prob}

Finally, from Propositions \ref{prop:comparison-1} and \ref{prop:comparison-2}, one may consider dealing with a realization result that involves the two parameters $\gp^-$ and $\mu^-$, which is the following problem.

\begin{prob}
Given any two integers $r,t\ge 2$, construct a graph $G$ such that $\gp^-(G)=r$ and $\mu^-(G)=t$.
\end{prob}


\section*{Acknowledgement}

We are thankful to two reviewers for numerous remarks that helped to improve the paper.
We want to thank Bal\'{a}zs Patk\'{o}s for pointing to us the theorem of Bollob\'{a}s and Wessel.
B.B. was supported by the Slovenian Research Agency (ARRS) under the grants P1-0297, J1-2452, J1-3002, and J1-4008. I.G.Y. has been partially supported by the Spanish Ministry of Science and Innovation through the grant PID2019-105824GB-I00.


\begin{thebibliography}{999}

\bibitem{BM} H.-J.~Bandelt, H.~M.~Mulder,  Distance-hereditary graphs, J.~Combin.\ Theory Ser.~B 41 (1986) 182--208.

\bibitem{bol}
B.~Bollob\'{a}s,
On a conjecture of Erd\H{o}s, Hajnal and Moon,
Amer.\ Math.\ Monthly 74 (1967) 178--179.

\bibitem{bls-book}
A.~Brandstaet, V.~B.~Le, J.~P.~Spinrad, {\em Graphs Classes: A Survey}, SIAM, Philadelphia, 1999.

\bibitem{ullas-2016}
U.~Chandran~S.V., G.~J.~Parthasarathy,
The geodesic irredundant sets in graphs, Int.\ J.\ Math.\ Combin.\  4 (2016) 135--143.

\bibitem{ChenChv}
X.~Chen, V.~Chv\' atal, Problems related to a De Bruijn-Erd\H{o}s theorem, Discrete Appl.\ Math.\ 156 (2008) 2101--2108.

\bibitem{chvatal-2018}
V.~Chv\' atal, A de Bruijn-Erd\H{o}s theorem in graphs? Graph theory--favorite conjectures and open problems. 2. Probl.\ Books in Math., Springer, Cham, (2018) 149--176.

\bibitem{CDDH}
S.~Cicerone, G.~Di~Stefano, L.~Dro\v{z}\dj ek,  J.~Hed\v{z}et, S.~Klav\v{z}ar, I.~G.~Yero, Variety of mutual-visibility in graphs, Theoret.\ Comput.\ Sci.\ 974 (2023) 114096.

\bibitem{CDK}
S.~Cicerone, G.~Di~Stefano, S.~Klav\v{z}ar, On the mutual-visibility in Cartesian products and in triangle-free graphs, Appl.\ Math.\ Comput.\ 438 (2023) 127619.

\bibitem{CDKY}
S.~Cicerone, G.~Di~Stefano, S.~Klav\v{z}ar, I.~G.~Yero, Mutual-visibility in strong products of graphs via total mutual-visibility, arXiv:2210.07835 [math.CO] (14 Oct 2022).

\bibitem{DiS} G.~Di~Stefano, Mutual visibility in graphs, Appl.\ Math.\ Comput.\ 419 (2022) 126850.

\bibitem{DiSKKTY}
G.~Di Stefano, S.~Klav\v zar, A.~Krishnakumar, J.~Tuite, I.~G.~Yero, Lower general position sets in graphs, arXiv:2306.09965 [math.CO]. (16 June 2023).

\bibitem{EHM} P.~Erd\H{o}s, A.~Hajnal, J.~W.~Moon,
A problem in graph theory,
Amer.\ Math.\ Monthly 71 (1964) 1107--1110.

\bibitem{GJ}
M.~R.~Garey, M.~R.~Johnson, Computers and Intractability: A Guide to the Theory of NP-Completeness, Freeman, New York (1979)


\bibitem{HaHeHe} T.~W.~Haynes, S.~T.~Hedetniemi, M.~A.~Henning, \emph{Domination in Graphs: Core Concepts} Series: Springer Monographs in Mathematics, Springer, Cham, 2022. ISBN-13: DOI 9783031094958.

\bibitem{Korner}
J.~K\"orner,
On the extremal combinatorics of the Hamming space, J.\ Comb.\ Theory Ser.\ A 71 (1995) 112--126.

\bibitem{Kuziak}
D.~Kuziak, J.~A.~Rodr\'{\i}guez-Vel\'{a}zquez, Total mutual-visibility in graphs with emphasis on lexicographic and Cartesian products, Bull.\ Malays.\ Math.\ Sci.\ Soc.\ In press (2023). DOI :10.1007/s40840-023-01590-3.

\bibitem{manuel-2018}
P.~Manuel, S.~Klav{\v z}ar,
A general position problem in graph theory, Bull.\ Aust.\ Math.\ Soc.\ 98 (2018) 177--187.

\bibitem{Pl} M.~D.~Plummer, Some covering concepts in graphs, J.\ Combin.\ Theory 8 (1970) 91--98.

\bibitem{rodriguez-2022}
J.~A.~Rodr\'{\i}guez-Vel\'{a}zquez, Universal lines in graphs, Quaest.\ Math.\ 45 (2022) 1485--1500.

\bibitem{TK} J.~Tian, S.~Klav\v{z}ar, Graphs with total mutual-visibility number zero and total mutual-visibility in Cartesian products, Discuss.\ Math.\ Graph Theory, To appear (2023).

\bibitem{wes} W.~Wessel,
\"{U}ber eine Klasse paarer Graphen. I. Beweis einer Vermutung von Erd\H{o}s, Hajnal und Moon,
Wiss.\ Z.\ Tech.\ Hochsch.\ Ilmenau 12 (1966) 253--256.

\bibitem{we} D.~B.~West, Introduction to Graph Theory (Second Edition), Prentice Hall, USA, 2001.



\end{thebibliography}
\end{document}